\documentclass[a4paper, 12pt]{amsart}
\usepackage[T1]{fontenc}
\usepackage[utf8]{inputenc}
\usepackage{paralist}
\usepackage{amsmath}
\usepackage{amsfonts}
\usepackage{amssymb}
\usepackage{amsthm}
\vfuzz \hfuzz
\newcommand{\N}{\mathbb{N}}

\newcommand{\Z}{\mathbb{Z}}
\newcommand{\m}{\mathrm{mod} \thinspace}

\newcommand{\Ext}{\mathrm{Ext}}

\newcommand{\GL}{\mathrm{GL}}

\newcommand{\Mat}{\mathrm{Mat}}
\newcommand{\diag}{\mathrm{diag}}

\newtheorem{theorem}{Theorem}[section]
\newtheorem*{theorem1}{Theorem}
\newtheorem{corollary}[theorem]{Corollary}
\newtheorem{lemma}[theorem]{Lemma}
\newtheorem{definition}[theorem]{Definition}
\newtheorem{proposition}[theorem]{Proposition}

\newtheorem{assumption}[theorem]{Assumption}

\begin{document}

\begin{abstract}
Using the general framework of polynomial representations defined by Doty and generalizing the definition given by Doty, Nakano  and Peters for $G = \GL_n$, we consider polynomial representations of $G_r T$ for an arbitrary closed reductive subgroup scheme $G \subseteq \GL_n$ and a maximal torus $T$ of $G$ in positive characteristic. We give sufficient conditions on $G$ making a classification of simple polynomial $G_r T$-modules similar to the case $G = \GL_n$ possible and apply this to recover the corresponding result for $\GL_n$ with a different proof, extending it to symplectic similitude groups, Levi subgroups of $\GL_n$ and, in a weaker form, to odd orthogonal similitude groups. We also consider orbits of the affine Weyl group and give a condition for equivalence of blocks of polynomial representations for $G_r T$ in the case $G = \GL_n$.
\end{abstract}

\title{On simple polynomial $G_r T$-modules}
\author{Christian Drenkhahn}
\maketitle

\section{Introduction}
When considering rational representations of $G = \GL_n$ over an infinite field, one can reduce many questions to polynomial representations. These representations correspond to representations of certain finite-dimensional algebras, the so-called Schur algebras (see for example \cite{G2, Martin1}). Letting $T$ be a maximal torus of $G$ and $k$ be an infinite field of positive characteristic, one can also study the category of $G_r T$-modules to get information on the representation theory of $G$ (see \cite[II.9]{Jantz1}).\\
\indent In \cite{DNP2}, Doty, Nakano and Peters combined these approaches and considered polynomial representations of $G_r T$ for $G = \GL_n$. They show that these representations correspond to representations of the so-called infinitesimal Schur algebras which are subalgebras of the ordinary Schur algebras. \\
\indent In \cite{Doty1}, Doty defined a general notion of polynomial representations and analogues of Schur algebras for algebraic groups contained in $\GL_n$. Using his definition, one can study polynomial representations of $G_r T$ for other algebraic groups $G$. \\
\indent If a reductive group $G \subseteq \GL_n$ admits a graded polynomial representation theory in the sense of \cite{Doty1}, it is natural to ask for which character $\lambda$ contained in the character group $X(T)$ of $T$ the simple $G_r T$-module $\widehat{L_r}(\lambda)$ is a polynomial $G_r T$-module. This problem was solved for $G = \GL_n$ in \cite{DNP2}, but in contrast to most results of that paper, the proof does not carry over to other reductive groups.\\
\indent Motivated by this problem, we develop a new technique using functions $\varphi: X(T) \rightarrow \Z^l$ with certain technical properties, see \ref{assumption: existence of function}. The existence of such a function implies a classification of simple polynomial $G_r T$-modules similar to the case $G = \GL_n$ (see Theorem \ref{thm: all weights of simple modules polynomial iff weight belongs to Xr(D)+p^r P(D)}, Corollary \ref{cor: criterion for simple modules with Phi and d}). We give a criterion for the existence of such a function (see Theorem \ref{thm: construction of phi}), allowing us to give a new proof of the result for $\GL_n$ and extend it to several other cases in a uniform way.\\
\indent    
Let $D = \overline{T}$ be the closure of $T$ in the algebraic monoid of $(n \times n)$-matrices, $P(D) \subseteq X(T)$ the set of polynomial weights, i.e. the character monoid of $D$, and $X_r(T) \subseteq X(T)$ the set of $p^r$-restricted weights. Then our main application can be summarized as follows:
\begin{theorem1}
\begin{enumerate}[(a)]
\item Let $G = GSp_{2l}$ be the symplectic dilation group, $T \subseteq G$ the torus of diagonal matrices. Then a complete set of representatives for the isomorphism classes of simple polynomial $G_r T$-modules is given by $\lbrace \widehat{L_r}(\lambda) \vert \lambda \in P_r(D)+ p^r P(D)\rbrace$, where $P_r (D)= \lbrace \lambda \in P(D) \cap X_r(T) \vert  \lambda - p^r d \notin P(D)\rbrace$ and $d$ is the character of $T$ given by $t \mapsto t_{11}t_{nn}$.
\item Let $G = GO_{2l+1}$ be the connected component of the odd orthogonal dilation group,  $T$ the torus of diagonal matrices in $G$. If $\lambda \in X_r(T) + p^r X(T)$ and all weights of the module $\widehat{L_r}(\lambda)$ are polynomial, then $\lambda \in P_r(D) + p^r P(D)$, where $P_r (D)= \lbrace \lambda \in P(D) \cap X_r(T) \vert  \lambda - p^r d \notin P(D)\rbrace$ and $d$ is the character of $T$ given by $t \mapsto t_{l+1}$. 
\item Let $n_i \in \N$ such that $n = \sum_{i=1}^{l} n_i$ and embed $G = \GL_{n_1} \times \ldots \times \GL_{n_l}$ into $\GL_n$ as a Levi subgroup in the obvious way. Let $T$ be the torus of diagonal matrices in $G$, $M_i = \{n_{1}+ \ldots + n_{i-1}+1, \ldots, n_{1}+ \ldots + n_{i}\}$ and $d_i = \sum_{j \in M_i} e_j$ for $1 \leq i \leq l$, where $e_1, \ldots, e_n$ is the standard basis of $X(T) \cong \Z^n$. Then a complete set of representatives for the isomorphism classes of simple polynomial $G_r T$-modules is given by $\{\widehat{L_r}(\lambda) \vert \lambda \in P_r(D)+p^r P(D)\}$, where $P_r(D) = \{\lambda \in X_r(T) \cap P(D) \vert \lambda - p^r d_i \notin P(D)$ for  $1 \leq i \leq l\}$. 

\end{enumerate}
\label{thm: introduction}
\end{theorem1}

By studying the intersection of orbits of the affine Weyl group with suitable sets of polynomial weights, we also show that certain shift functors induce Morita equivalences between blocks of polynomial representations of $G_r T$ for $G = \GL_n$.       
  
\section{Notation and prerequisites}
In this section, we fix notation and recall some basic results on polynomial representations. For algebraic groups and $G_r T$, we follow the notation from \cite{Jantz1}. For an introduction to algebraic monoids, we refer the reader to \cite{Ren1}.\\
\indent Let $k$ be an infinite perfect field of characteristic $p > 0$ and $\Mat_n$ be the monoid scheme of $(n \times n)$-matrices over $k$. Let $d \in \N_0$ and denote by $A(n, d)$ the space generated by all homogeneous polynomials of degree $d$ in the coordinate ring $k[\Mat_n] = k[X_{ij}\vert 1 \leq i,j \leq n]$ of $\Mat_n$. Then $k[\Mat_n] = \bigoplus_{d \geq 0} A(n, d)$ is a graded $k$-bialgebra with comultiplication $\Delta: k[\Mat_n] \rightarrow k[\Mat_n] \otimes_k k[\Mat_n]$ and counit $\epsilon: k[\Mat_n] \rightarrow k$ given by 
\begin{equation*}
\Delta(X_{ij}) = \sum_{l=1}^n X_{il} \otimes_k X_{lj}, \quad \epsilon(X_{ij}) = \delta_{ij}
\end{equation*}
and each $A(n, d)$ is a subcoalgebra.\\
\indent As $\GL_n$ is dense in $\Mat_n$, the canonical map $k[\Mat_n] \rightarrow k[\GL_n]$ is injective, so that $k[\Mat_n]$ and each $A(n, d)$ can be viewed as a subcoalgebra of $k[\GL_n]$.\\
\indent Now let $G$ be a closed subgroup scheme of $\GL_n$ and $\pi: k[\GL_n] \rightarrow k[G]$ the canonical projection. Set $A(G) = \pi(k[\Mat_n])$ and $A_d(G)= \pi(A(n, d))$ as well as $ S_d(G) = A_d(G)^*$. Since $\pi$ is a homomorphism of Hopf algebras, $A(G)$ is a subbialgebra of $k[G]$ and each $A_d(G)$ is a subcoalgebra of $A(G)$. As $A_d(G)$ is finite dimensional, $S_d(G)$ is a finite dimensional associative algebra in a natural way. \\
\indent For $G = \GL_n$ and $T$ a maximal torus of $G$, the algebra $S_d(G)$ is the ordinary Schur algebra, while $S_d(G_r T)$ is the infinitesimal Schur algebra defined in \cite{DNP2}.\\ \indent Following \cite{Doty1}, we give the following definition.
\begin{definition}
\begin{enumerate}[(1)]
\item We say that $G$ admits a graded polynomial representation theory if the sum $\sum_{d \geq 0} A_d(G)$ is direct.
\item  We say that a rational $G$-module $V$ is a polynomial $G$-module if the corresponding comodule map $V \rightarrow V \otimes k[G]$ factors through $V \otimes A(G)$. 
\item If the comodule map $V \rightarrow V \otimes k[G] $ factors through $V \otimes A_d(G)$ for some $d \in \N_0$, we say that $V$ is homogeneous of degree $d$. 
\end{enumerate}
\end{definition}
Clearly, every $A_d(G)$-comodule is an $A(G)$-comodule and every $A(G)$-comodule is a $G$-module in a natural way. We record some other properties of polynomial representations from \cite[1.2]{Doty1} in the following proposition.

\begin{proposition}
\begin{enumerate}[(1)]
\item Suppose $G$ admits a graded polynomial representation theory and $V$ is a polynomial $G$-module. Then $V = \bigoplus_{d \geq 0} V_d$, where each $V_d$ is homogeneous of degree $d$. Furthermore, the category of homogeneous modules of degree $d$ is equivalent to the category of $S_d(G)$-modules.
\item If $G$ contains the center $Z(\GL_n)$ of $\GL_n$, then $G$ admits a graded polynomial representation theory. 
\end{enumerate}
\end{proposition}

As $A(G)$ is a factor bialgebra of $k[\Mat_n]$, it corresponds to a closed submonoid $M$ of $\Mat_n$. Since $A(G) \subseteq k[G]$ is a subbialgebra, $G \subseteq M$ is a dense subscheme, so that $M$ is the closure of $G$ in $\Mat_n$. Hence polynomial representations of $G$ can be regarded as rational representations of the algebraic monoid scheme $M = \overline{G}$. \\
\indent If $G = \GL_n$, a comparison of coordinate rings shows $\overline{G_r T} = M_r D$, where $M_r D = (F^r)^{-1}(D)$ with $F^r$ the $r$-th iteration of the Frobenius homomorphism, see \cite{DNP2}. \\
This can fail for general $G$: If $G$ is the full orthogonal similitude group of a 2-dimensional vector space with symmetric non-degenerate bilinear form and $T$ the torus of diagonal matrices in $G$, then $G_r T = T$, but $M_r D$ is strictly larger than $D = \overline{T}$. However, we always have $\overline{G_r T} \subseteq M_r D$. We do not know whether $M_r D = \overline{G_r T}$ e.g. for connected $G$.

\section{Simple polynomial $G_r T$-modules}

In this section, let $G \leq \GL_n$ be a closed connected reductive subgroup scheme containing $Z(\GL_n)$, $T \subseteq G$ be a maximal torus contained in a maximal torus $T' \subseteq \GL_n$. This is not really a restriction: It is well-known that any affine algebraic group can be embedded into $\GL_n$ as a closed subgroup for some $n \in \N$. If $G \subseteq \GL_n$ does not contain $Z(\GL_n)$, we can pass to the group $\langle G, Z(\GL_n)\rangle$ to get a reductive group containing $Z(\GL_n)$.\\
\indent Let $M$ be the closure of $G$ in $\Mat_n$, $D$ be the closure of $T$, $D'$ be the closure of $T'$ and $P(D) \subseteq X(T)$ be the character monoid of $D$, i.e. the set of polynomial weights, $W$ the Weyl group of $G$. Let $R$ be the root system of $G$ with respect to $T$ and $S \subseteq R$ a set of simple roots. Denote by $\langle- , -\rangle: X(T) \times Y(T) \rightarrow \Z$ the canonical perfect pairing of $X(T)$ and the cocharacter group $Y(T)$. For $\alpha \in R$, let $\alpha^{\vee}$ be the coroot of $\alpha$. Let $X_0(T) = \{\lambda \in X(T) \vert \langle\lambda, \alpha^{\vee}\rangle = 0$ for all $ \alpha \in R\}$ as well as $X_r(T) = \{\lambda \in X(T) \vert 0 \leq \langle\lambda, \alpha^{\vee}\rangle \leq p^r-1$ for all $\alpha \in S\}$.\\
\indent We want to determine the simple polynomial $G_r T$-modules. The isomorphism classes of simple $G_r T$-modules are parametrized by $X(T)$ via  $\lambda \mapsto \widehat{L_r}(\lambda)$, where $\widehat{L_r}(\lambda)$ has highest weight $\lambda$. For convenience and further reference, we collect some basic properties of the $\widehat{L}_r(\lambda)$ from \cite[II.9.6]{Jantz1} in the form we need in the following proposition. 
\begin{proposition}[\cite{Jantz1}]
\begin{enumerate}[(a)]
\item For each $\lambda \in X(T)$, there is a simple $G_r T$-module $\widehat{L}_r(\lambda)$ with $\dim_k (\widehat{L}_r(\lambda))_{\lambda} = 1$. Each weight $\mu$ of $\widehat{L}_r(\lambda)$ satisfies $\mu \leq \lambda$ with respect to the standard ordering on $X(T)$. 
\item Each simple $G_r T$-module is isomorphic to exactly one $\widehat{L}_r(\lambda)$ with $\lambda \in X(T)$.
\item Each $\widehat{L}_r(\lambda)$ is isomorphic to $L_r(\lambda)$ as a $G_r$-module.
\item There are isomorphisms of $G_r T$-modules 
\begin{equation*}
\widehat{L}_r(\lambda+ p^r \mu) \cong \widehat{L}_r(\lambda) \otimes_k p^r \mu
\end{equation*} 
for all $\lambda, \mu \in X(T)$.
\item If $\lambda \in X_r(T)$, then $\widehat{L}_r(\lambda) \cong L(\lambda) \vert_{G_r T}$, where $L(\lambda)$ is the simple $G$-module of highest weight $\lambda$. 
\end{enumerate}
\label{prop: properties of the Lr(lambda)}
\end{proposition}
As a motivation for our approach, consider the case where $G = \GL_n$ and $T \leq G$ is the maximal torus of diagonal matrices.\\
\indent According to \cite[Section 3]{DNP2}, the simple $G_r T$-module $\widehat{L_r}(\lambda)$ is a polynomial $G_r T$-module iff $\lambda \in P_r(D) + p^r P(D)$, where 
\begin{align*}
P_r(D) = \{ &\lambda \in P(D) \vert  0 \leq \lambda_i - \lambda_{i+1} \leq p^r -1 \text{ for } 1 \leq i \leq n-1,\\ &0 \leq \lambda_n \leq p^r -1\}.
\end{align*}
A character $\lambda \in X(T)$ belongs to $P(D)$ iff $\lambda_i \geq 0$ for $1 \leq i \leq n$, that is, iff $\min \{\lambda_i \vert 1 \leq i \leq n\} \in \N_0$ where $\N_0 = \N \cup \{0\}$. Taking this minimum is compatible with multiplication by natural numbers, in particular with multiplication by $p^r$.\\
\indent Given characters $\lambda, \mu \in X(T)$, we can of course find a permutation $w \in W \cong S_n$ such that this minimum is attained at the same coordinate for $w.\lambda$ and $\mu$, so that 
\begin{equation*}
 \min\{(w.\lambda+\mu)_i \vert 1 \leq i \leq n\} = \min \{\lambda_i \vert 1 \leq i \leq n\} + \min \{\mu_i \vert 1 \leq i \leq n\}.
\end{equation*}
Writing $d = \det \vert_{T} = (1, \ldots, 1)$, we have 
\begin{equation*} 
X_0(T) = \langle d \rangle, \quad \min\{ d_i \vert 1 \leq i \leq n\} = 1
\end{equation*}
and we can write 
\begin{equation*}
P_r(D) = \{ \lambda \in X_r(T) \cap P(D) \vert \lambda - p^r d \notin P(D)\}. 
\end{equation*}
We axiomatize these properties of $\det \vert_T$ and the function $X(T) \rightarrow \Z, \lambda \mapsto \min\{\lambda_i \vert 1 \leq i \leq n\}$ for general $G$ in the following

\begin{assumption}
Suppose there exists a function $\varphi: X(T) \rightarrow \Z^l$ with the following properties:
\begin{enumerate}[(1)]
\item $\forall \lambda \in X(T): \varphi(\lambda) \in \N_0^l \Longleftrightarrow \lambda \in P(D)$,
\item $\forall \lambda \in X(T):\varphi(p^r \lambda) =  p^r \varphi(\lambda)$,
\item $\forall \lambda, \lambda' \in X(T): \exists w \in W: \varphi(w.\lambda + \lambda') = \varphi(\lambda) + \varphi(\lambda')$,
\item $\varphi \vert_{X_0(T)}: X_0(T) \rightarrow \Z^l$ is bijective.
\end{enumerate}
\label{assumption: existence of function}
\end{assumption}

Since $W$ acts trivially on $X_0(T)$, $(1),(3)$ and $(4)$ imply that the restriction to $X_0(T)$ of any function $\varphi$ as in \ref{assumption: existence of function} is an isomorphism mapping $P(D)\cap X_0(T)$ to $\N_0^l$.\\
\indent It will turn out that \ref{assumption: existence of function} is sufficient for a classification similar to the case $G = \GL_n$. We first show that such a function is essentially unique.

\begin{proposition}
If  $\varphi, \psi: X(T) \rightarrow \Z^l$ are functions as in \ref{assumption: existence of function}, there is a permutation of coordinates $\sigma: \Z^l \rightarrow \Z^l$ such that $\varphi = \sigma \circ \psi$. 
\label{prop: unicity of phi}
\end{proposition}

\begin{proof}
Let $e_1, \ldots, e_l$ be the standard basis of $\Z^l$ and $d_1, \ldots, d_l$ resp. $d'_1, \ldots, d'_l$ be the preimages of the $e_i$ in $X_0(T)$ for $\varphi$ resp. $\psi$. Then every element of $X_0(T) \cap P(D)$ can be written as a linear combination in the $d_i$ with non-negative coefficients. Writing the $d'_i$ as such a linear combination and then writing the $d_i$ as such a linear combination in the $d'_i$, we see that the base change matrix for the two bases is a permutation matrix. Thus, there is $\sigma \in S_l$ such that $d_i = d'_{\sigma(i)}$ for $1 \leq i \leq l$.\\
\indent Now let $\lambda \in X(T)$ and set $\tilde{\lambda} = \lambda - \sum_{i=1}^l \varphi(\lambda)_i d_i$. Using $(3)$ and that $W$ acts trivially on $X_0(T)$ for the first equality and the fact that $\psi \vert_{X_0(T)}$ is an isomorphism for the second equality, we get 
\begin{align*}
\psi(\tilde{\lambda}) &= \psi(\lambda) + \psi(- \sum_{i=1}^l \varphi(\lambda)_i d_i)= \psi(\lambda) - \sum_{i=1}^l \varphi(\lambda)_i \psi(d_i) \\
&= \psi(\lambda) - \sum_{i=1}^l \varphi(\lambda)_{i} \psi(d'_{\sigma(i)}) = \psi(\lambda)-\sum_{i=1}^l \varphi(\lambda)_i e_{\sigma(i)}.
\end{align*} 
Applying the same arguments to $\varphi$, we get $\varphi(\tilde{\lambda}) = 0$ and $\varphi(\tilde{\lambda} - d_i) = -e_i$ for $1 \leq i \leq l$. Thus, $\tilde{\lambda}\in P(D)$ and $\tilde{\lambda}-d_i = \tilde{\lambda} - d'_{\sigma(i)} \notin P(D)$ for $1 \leq i \leq l$, so that $\psi(\tilde{\lambda}) \in \N_0^l$ and $\psi(\tilde{\lambda}) - e_i \notin \N_0^l$. This shows $\psi(\tilde{\lambda}) = 0$, so that $\psi(\lambda)= \sum_{i=1}^l \varphi(\lambda)_i e_{\sigma(i)}$, hence $\varphi(\lambda)_i = \psi(\lambda)_{\sigma(i)}$.    
\end{proof}

\begin{definition}
Suppose $\varphi$ is a function as in \ref{assumption: existence of function} for $G$. Letting $e_1, \ldots, e_l$ be the standard basis of $\Z^l$ and $d_1, \ldots, d_l$ be the preimages of the $e_i$ in $X_0(T)$ with respect to $\varphi$, we set \begin{equation*}
P_r(D) = \lbrace \lambda \in P(D) \cap X_r(T) \vert \forall i \in \{1, \ldots, l\}: \lambda - p^r d_i \notin P(D) \rbrace.
\end{equation*}
\end{definition}

It follows from the proof of \ref{prop: unicity of phi} that $P_r(D)$ does not depend on the choice of $\varphi$ or the $d_i$. By the remarks preceding \ref{assumption: existence of function}, we recover the original definition of $P_r (D)$ in the case $G = \GL_n$.

\begin{lemma}
Let $\lambda = \lambda_0 + p^r \tilde{\lambda}= \lambda_0' + p^r \tilde{\lambda'}$ with $\lambda_0, \lambda_0' \in X_r(T), \tilde{\lambda}, \tilde{\lambda'} \in X(T)$. Then there is $\mu \in X_0(T)$ such that $\lambda_0 = \lambda_0' + p^r \mu, \tilde{\lambda}' = \tilde{\lambda} + \mu$.
\label{lem: decompositions of lambda differ by an element of X_0(T)}
\end{lemma}

\begin{proof}
We have $\lambda_0 - \lambda_0' = p^r \mu$, where $\mu = \tilde{\lambda}' - \tilde{\lambda}$. For all $\alpha \in R$, we have
\begin{equation*}
p^r \langle\mu, \alpha^\vee\rangle = \langle\lambda_0 - \lambda_0', \alpha^\vee\rangle = \langle\lambda_ 0, \alpha^\vee \rangle -  \langle\lambda_0', \alpha^\vee\rangle.
\end{equation*}
As $\lambda_0, \lambda_0' \in X_r(T)$, we have
\begin{equation*} 
-(p^r - 1) \leq \langle\lambda_ 0, \alpha^\vee \rangle -  \langle\lambda_0', \alpha^\vee\rangle \leq p^r -1, 
\end{equation*}
forcing $\langle\mu, \alpha^\vee\rangle = 0$. Thus, $\mu \in X_0(T)$ and the result follows.
\end{proof}

\begin{theorem}
Suppose \ref{assumption: existence of function} holds for $G$.
If $\lambda \in X_r(T) + p^r X(T)$, then there is a unique decomposition $\lambda = \lambda_0 + p^r \tilde{\lambda}$ with $\lambda_0 \in P_r(D), \tilde{\lambda} \in X(T)$. Furthermore, if all weights of the simple $G_r T$-module $\widehat{L_r}(\lambda)$ belong to $P(D)$, then $\lambda \in P_r(D) + p^r P(D)$.
\label{thm: all weights of simple modules polynomial iff weight belongs to Xr(D)+p^r P(D)}
\end{theorem}

\begin{proof}
Let $\lambda = \lambda_0 + p^r \tilde{\lambda}$ with $\lambda_0 \in X_r(T), \tilde{\lambda} \in X(T)$.
As $d_i \in X_0(T)$, we get that for each $a \in \Z^l$, $\lambda_0 + p^r \sum_{i=1}^l a_i d_i + p^r (\tilde{\lambda} - \sum_{i=1}^l a_i d_i)$ is another decomposition such that $\lambda_0 +  p^r \sum_{i=1}^l a_i d_i \in X_r(T)$. Since $d_1, \ldots, d_l$ generate $X_0(T)$, \ref{lem: decompositions of lambda differ by an element of X_0(T)} shows that every decomposition of $\lambda$ arises in this fashion. By $(1), (3), (4)$ and since $W$ acts trivially on $X_0(T)$, there is a unique $a \in \Z^l$ such that $\lambda_0 +  p^r \sum_{i=1}^l a_i d_i \in P_r(D)$. Thus, there is a unique decomposition $\lambda = \lambda_0 + p^r \tilde{\lambda}$ such that $\lambda_0 \in P_r (D)$. \\
\indent We show by contraposition that if all weights of $\widehat{L_r}(\lambda)$ are polynomial, then $\tilde{\lambda} \in P(D)$. Suppose that $\tilde{\lambda} \notin P(D)$, so that $\varphi(\tilde{\lambda}) \notin \N_0^l$. Then there is $i \in \{1, \ldots, l\}$ such that $\varphi(\tilde{\lambda})_i < 0$. Since $\lambda_0 - p^r d_i \notin P(D)$, $(1),(3),(4)$ yield $\varphi(\lambda_0)_i < p^r$. Using $(2)$ and $(3)$, we see that there is $w \in W$ such that \begin{equation*}
\varphi(w. \lambda_0 + p^r \tilde{\lambda})_i = \varphi(\lambda_0)_i + p^r \varphi(\tilde{\lambda})_i < 0,
\end{equation*}
so that $w. \lambda_0 + p^r \tilde{\lambda} \notin P(D)$ by $(1)$. \\
\indent Since $\widehat{L_r}(\lambda_0)$ lifts to a simple $G$-module by \ref{prop: properties of the Lr(lambda)}, its character is $W$-invariant, so that $w. \lambda_0$ is a weight of $\widehat{L_r}(\lambda_0)$ and $w. \lambda_0 + p^r \tilde{\lambda}$ is a weight of $\widehat{L_r}(\lambda_0 + p^r \tilde{\lambda}) \cong \widehat{L_r}(\lambda_0) \otimes_k p^r \tilde{\lambda}$, so that not all weights of $\widehat{L_r}(\lambda)$ are polynomial.
\end{proof}

For $G = \GL_n$, it was shown in the Appendix of \cite{Nak1} that a $G_r T$-module such that all of its weights are polynomial is a polynomial $G_r T$-module. It is not known for which general $G$ a similar statement holds. However, if $M$ is normal as a variety, we can at least prove this for some simple modules. 

\begin{proposition}
Suppose that $M$ is a normal variety. If $\lambda = \lambda_0 + p^r \tilde{\lambda} \in (X_r(T)\cap P(D)) + p^r P(D)$, the simple $G_r T$-module $\widehat{L}_r(\lambda)$ is a polynomial $G_r T$-module.  
\label{prop: M normal => certain simple GrT-modules are polynomial}
\end{proposition}

\begin{proof}
We have $\widehat{L}_r(\lambda) \cong \widehat{L_r}(\lambda_0) \otimes_k p^r \tilde{\lambda}$ and $\widehat{L_r}(\lambda_0)$ lifts to $G$ by \ref{prop: properties of the Lr(lambda)}. Since $M$ is normal and $\lambda_0$ is a dominant weight contained in $P(D)$, $\widehat{L_r}(\lambda_0)$ lifts to $M$ by \cite[3.5]{Doty2}, so that $\widehat{L}_r(\lambda_0)$ is a polynomial $G_r T$-module.\\
\indent As $\tilde{\lambda}$ is a $D$-module, $p^r \tilde{\lambda}$ is an $F^{-r} (D)$-module, where $F^r$ is the $r$-th iteration of the Frobenius morphism, and thus a polynomial $G_r T$-module by restriction. Hence $\widehat{L}_r(\lambda)$ is a polynomial $G_r T$-module. 
\end{proof}

\begin{corollary}
Let $M$ be normal and suppose that \ref{assumption: existence of function} holds for $G$.
If $\lambda \in X_r(T) + p^r X(T)$, then the simple module $\widehat{L_r}(\lambda)$ is a polynomial $G_r T$-module iff $\lambda \in P_r(D) + p^r P(D)$.
\end{corollary}

\begin{proof}
This follows directly from \ref{thm: all weights of simple modules polynomial iff weight belongs to Xr(D)+p^r P(D)} and \ref{prop: M normal => certain simple GrT-modules are polynomial}.
\end{proof}

\begin{corollary}
Let the derived subgroup of $G$ be simply connected and $M$ be normal. Suppose \ref{assumption: existence of function} holds for $G$. Then a system of representatives of isomorphism classes of simple polynomial $G_r T$-modules is given by $\lbrace \widehat{L_r}(\lambda) \vert \lambda \in P_r(D)+ p^r P(D)\rbrace$.
\label{cor: criterion for simple modules with Phi and d}
\end{corollary}

\begin{proof}
Let $\lambda \in X(T)$. Since the derived subgroup of $G$ is simply connected, $X_r(T)$ contains a set of representatives of $X(T)/p^r X(T)$, so that we have $\lambda \in X_r(T)+ p^r X(T)$. \\
\indent By \ref{thm: all weights of simple modules polynomial iff weight belongs to Xr(D)+p^r P(D)}, all weights of $\widehat{L_r}(\lambda)$ are polynomial iff $\lambda \in P_r(D) + p^r P(D)$ and by \ref{prop: M normal => certain simple GrT-modules are polynomial}, $\widehat{L_r}(\lambda)$ is a polynomial $G_r T$-module in this case. Since all weights of polynomial $G_r T$-modules are polynomial and a polynomial module is a simple polynomial $G_r T$-module iff it is simple as a $G_r T$-module, the result follows.      
\end{proof}

We now give an explicit construction for a function as in \ref{assumption: existence of function} under certain assumptions.

\begin{theorem}
Suppose there is an action of $W$ on $X(T')$ such that the canonical projection $\pi: X(T') \rightarrow X(T)$ is $W$-equivariant and there are $b_1, \ldots, b_s \in \pi^{-1}(X_0(T)) \cap P(D')$ such that 
\begin{enumerate}[(a)]
\item $\forall i \in \{1, \ldots, s\},j \in \{1, \ldots, n\}: (b_i)_j \in \{ 0, 1\}$,
\item the sets $M_{i}=\{ j \in \{1, \ldots, n\} \vert (b_i)_j = 1\}$ form a partition of $\{1, \ldots, n\}$,
\item the image of $W$ in $S_{X(T')}$ is contained in the image of $W' \cong S_n$ and contains the subgroup $S_{M_1} \times \ldots \times S_{M_s}$, 
\item $\exists d_1, \ldots, d_l \in \{b_1, \ldots, b_s \}: \pi(d_1), \ldots, \pi(d_l)$ is a basis of  $X_0(T)$ such that each $\pi(b_i)$ is a linear combination of the $\pi(d_i)$ with non-negative coefficients. 
\end{enumerate}
Then there is a function $\varphi: X(T) \rightarrow \Z^l$ with the properties of \ref{assumption: existence of function}.
\label{thm: construction of phi} 
\end{theorem}

\begin{proof}
For every $i \in \{1, \ldots, s\}$, write $\pi(b_i) = \sum_{j=1}^l n_{ij}\pi(d_j)$ with $n_{ij} \in \N_0$.
Define \begin{equation*}
\varphi: X(T')\rightarrow \Z^l, \quad \lambda \mapsto  \sum_{j=1}^l \sum_{i=1}^s \min\{ \lambda_a \vert a \in M_i \} n_{ij} e_j.
\end{equation*}
We show that $\varphi$ induces a function $\bar{\varphi}: X(T) \rightarrow \Z^l$ with the properties of \ref{assumption: existence of function}. For this, we show the following statements $(i)-(v)$.\\
\indent $(i)$ $\forall \lambda \in X(T'), \mu \in \ker(\pi): \varphi(\lambda + \mu) = \varphi(\lambda):$\\
Let $\mu \in \ker(\pi)$. Since the sets $M_i$ are pairwise disjoint, we have $\mu' = \mu - \sum_{i=1}^s \min\{ \mu_a \vert a \in M_i\}b_i \in P(D')$ and for every $i$, there is $a \in M_i$ such that $\mu'_a = 0$. By \ref{lem: coordinates in Mi of elements of T'bot} below and $(a)$, we get $\mu'_l = 0$ for all $l \in M_i$. As the $M_i$ form a partition of $\{1, \ldots, n\}$, we get $\mu' = 0$, so that $ \mu = \sum_{i=1}^s \min\{ \mu_a \vert a \in M_i\}b_i$. Consequently, $\varphi(\lambda + \mu) = \varphi(\lambda) + \varphi(\mu)$ for all $\lambda \in X(T')$.\\
\indent As $\mu \in \ker(\pi)$, we have 
\begin{align*}
0 = \pi(\mu)=& \sum_{i=1}^s \min\{\mu_a \vert a \in M_i\} \pi(b_i) \\
	=& \sum_{j=1}^l \left(\sum_{i=1}^s n_{ij} \min\{\mu_a \vert a \in M_i\}\right)\pi(d_j),
\end{align*}
so that $\sum_{i=1}^s n_{ij} \min\{\mu_a \vert a \in M_i\} = 0$ for every $j$ and $\varphi(\mu)= 0$. Hence $(i)$ holds.\\  
\indent $(ii)$ $\forall \lambda \in X(T'): \varphi(p^r \lambda) = p^r \varphi(\lambda):$\\
Clear by definition.\\
\indent $(iii)$ $\forall \lambda \in X(T'): \varphi(\lambda) \in \N_0^l \Longleftrightarrow \lambda \in P(D)+ \ker(\pi)$:\\ Suppose $\varphi(\lambda) \in \N_0^l$. We have $\lambda - \sum_{i=1}^s \min\{\lambda_a \vert a \in M_i \}b_i \in P(D')$, so that $\pi(\lambda) - \sum_{i=1}^s \min\{\lambda_a \vert a \in M_i \} \pi(b_i) \in \pi(P(D'))$. Hence $\pi(\lambda) \in \pi(P(D')) +  \sum_{j=1}^l\sum_{i=1}^s \min\{\lambda_a \vert a \in M_i \}n_{ij}\pi(d_j)$. By assumption, the coefficients in the linear combination are non-negative, so that this set is contained in $\pi(P(D'))$ and $\lambda \in P(D') + \ker(\pi)$.\\
\indent Now let $\lambda = \lambda' + \mu \in P(D') + \ker(\pi)$. Since $\varphi(\lambda) = \varphi(\lambda')$ by $(i)$ and since $\lambda'$ has no negative coordinates and $n_{ij} \geq 0$, we have $\varphi(\lambda) \in \N_0^l$.\\  
\indent $(iv)$ $\forall \lambda, \lambda' \in X(T'): \exists w \in W: \varphi(w .\lambda + \lambda') = \varphi(\lambda) + \varphi(\lambda')$:\\
Using $(c)$, we permute the coordinates in each $M_i$ for $\lambda$ in such a way that the lowest coordinate for $w.\lambda$ and $\lambda'$ in each $M_i$ occurs at the same place. Then clearly $\varphi(w.\lambda + \lambda') = \varphi(\lambda) + \varphi(\lambda')$.\\
\indent $(v)$ $\varphi(d_i) = e_i$:\\
This follows directly from $(d)$ and the definition of $\varphi$.\\

\indent The statement $(i)$ implies that there is a function $\bar{\varphi}: X(T) \rightarrow \Z^l$ such that $\varphi = \bar{\varphi} \circ \pi$. Now $(ii)-(v)$ imply that $\bar{\varphi}$ has the properties of \ref{assumption: existence of function}.   

\end{proof}

\begin{lemma}
In the situation of \ref{thm: construction of phi}, let $i \in \{1, \ldots, s\}$ and $\mu \in \ker(\pi)$. Then $\mu_k = \mu_l$ for all $l, k \in M_i$.
\label{lem: coordinates in Mi of elements of T'bot}
\end{lemma}

\begin{proof}
Let  $\mu \in \ker(\pi)$ and $i \in \{1, \ldots, s\}$. Suppose there are $k, l \in M_{i}$ such that $\mu_k \neq \mu_l$. Let $w = (k \ l) \in W$ be the transposition with respect to $k$ and $l$. Then $\mu - w.\mu= (\mu_k - \mu_l)(e_k - e_l)\in \ker(\pi)$. As the restriction of this element is zero in $X(T)$ and $X(T)$ is torsion-free, $(e_k - e_l)\vert_T = 0$, so that $e_k \vert_{T} = e_l \vert_{T}$. It follows that $w$ acts trivially on $T$, a contradiction.
\end{proof}

The function $\varphi$ defined in the proof of \ref{thm: construction of phi} provides an easy way to check whether the restriction of a character $\lambda \in X(T')$ to $T$ is a polynomial weight.\\
We now apply our results to several examples. The first part of the following corollary is \cite[Corollary 3.2]{DNP1}. For the orthogonal and symplectic similitude groups $GO_{2l+1}$ and $GSp_{2l}$, we adopt the definition and notation from \cite[Sections, 5, 7]{Doty1}. Recall that with suitable choice of defining form, the maximal torus of diagonal matrices in $GSp_{2l}$ resp. $GO_{2l+1}$ is given by $\{\diag(t_1, \ldots, t_{2l}) \vert t_{ii}t_{i'i'} = t_{jj}t_{j'j'} $ for $1 \leq i,j \leq l \}$ resp. $\{\diag(t_1, \ldots, t_{2l+1}) \vert t_{ii}t_{i'i'} = t_{jj}t_{j'j'} $ for $1 \leq i,j \leq l+1 \}$.  

\begin{corollary}
\begin{enumerate}[(a)]
\item For $G = \GL_n, T$ the torus of diagonal matrices, a complete set of representatives of simple polynomial $G_r T$-modules is given by $\lbrace \widehat{L_r}(\lambda) \vert \lambda \in P_r(D)+ p^r P(D)\rbrace$, where $P_r (D)= \lbrace \lambda \in P(D) \vert 0 \leq \lambda_i - \lambda_{i+1} \leq p^r - 1$ for  $ 1 \leq i \leq n - 1, 0 \leq \lambda_n \leq p^r - 1\rbrace$. 
\item Let $G = GSp_{2l}, T$ the torus of diagonal matrices in $G$. Then a complete set of representatives of the isomorphism classes of simple polynomial $G_r T$-modules is given by $\lbrace \widehat{L_r}(\lambda) \vert \lambda \in P_r(D)+ p^r P(D)\rbrace$, where $P_r (D)= \lbrace \lambda \in P(D) \cap X_r(T) \vert  \lambda - p^r d \notin P(D)\rbrace$ and $d$ is the character of $T$ given by $t \mapsto t_{11}t_{nn}$.
\item Let $G = GO_{2l+1}, T$ the torus of diagonal matrices in $G$. If $\lambda \in X_r(T) + p^r X(T)$ and all weights of the module $\widehat{L_r}(\lambda)$ are polynomial, then $\lambda \in P_r(D) + p^r P(D)$, where $P_r (D)= \lbrace \lambda \in P(D) \cap X_r(T) \vert  \lambda - p^r d \notin P(D)\rbrace$ and $d$ is the character of $T$ given by $t \mapsto t_{l+1}$. 
\item Let $n_i \in \N$ such that $n = \sum_{i=1}^{l} n_i$ and embed $G = \GL_{n_1} \times \ldots \times \GL_{n_l}$ into $\GL_n$ as a Levi subgroup in the obvious way. Let $M_i = \{n_{1}+ \ldots + n_{i-1}+1, \ldots, n_{1}+ \ldots + n_{i}\}$ and $d_i = \sum_{j \in M_i} e_j$ for $1 \leq i \leq l$. Then a complete set of representatives of isomorphism classes of simple polynomial $G_r T$-modules is given by $\{\widehat{L_r}(\lambda) \vert \lambda \in P_r(D)+p^r P(D)\}$, where $P_r(D) = \{\lambda \in X_r(T) \cap P(D) \vert \lambda - p^r d_i \notin P(D)$ for  $1 \leq i \leq l\}$.

\end{enumerate}
\label{cor: simple polynomial GrT-modules for GL_n, GSp_2l, GO2_l+1}
\end{corollary}

\begin{proof}
(a) Let $d = \det \vert_{T} = (1, \ldots, 1), M = \{1, \ldots, n\} $. Clearly $d$ and $M$ have the properties of \ref{thm: construction of phi} and $P_r(D)$ has the required form. As the derived subgroup of $\GL_n$ is simply connected and $\Mat_n$ is normal, the result follows from \ref{cor: criterion for simple modules with Phi and d}.
\\
(b) We have $W \cong S_{l} \ltimes (\Z/(2))^l$. We can identify $W$ with a subgroup of $W' \cong S_{2l}$ by letting $\sigma \in S_l$ act on $\lambda \in X(T')$ via $(\sigma.\lambda)_{i} = \lambda_{\sigma^{-1}(i)}, (\sigma.\lambda)_{i'} = \lambda_{(\sigma^{-1}(i))'}$ for $1 \leq i \leq l$ and letting the $i$-th unit vector of $(\Z/(2))^l$ act as the transposition $(i i')$. With this action, the canonical projection $\pi: X(T') \rightarrow X(T)$ is $W$-equivariant.\\
\indent For $1 \leq i \leq l$, let $d_i = e_i + e_{i'}$ and $M_i = \{i, i'\}$. Then $\pi(d_{i}) = \pi(d_{j})$ for $1 \leq i, j \leq l$,the character $d =\pi(d_1)$ is a basis of $X_0(T)$ and the $M_i$ form a partition of $\{1, \ldots, 2l\}$. \\
\indent Inside $W \cong S_l \ltimes (\Z/(2))^l$, $(\Z/(2))^l$ corresponds to $S_{M_1} \times \ldots \times S_{M_l}$. Thus, \ref{thm: construction of phi} yields that $G$ has the properties of \ref{assumption: existence of function} and $P_r(D)$ has the desired form.\\
\indent The derived subgroup of $G$ is $Sp_{2l}$, hence simply connected, and by the remark below 2.11 in \cite{DotyHun1}, $\overline{G}$ is normal. The result now follows from \ref{cor: criterion for simple modules with Phi and d}.   
\\
(c) For $1 \leq i \leq l$, let $d_i = e_i + e_{i'}$, $M_i = \{i, i'\}$ and let $d_{l+1}= e_{l+1}$, $M_{l+1} = \{l+1\}$. Then the $M_i$ form a partition of $\{1, \ldots, 2l+1\}$, the character $d =\pi(d_{l+1})$ is a basis of $X_0(T)$ and $\pi(d_i) = 2 \pi(d)$ for $1 \leq i \leq l$. \\
\indent The other properties of \ref{thm: construction of phi} can be checked as in the proof of $(b)$, so that $G$ has the properties of \ref{assumption: existence of function} and $P_r(D)$ has the desired form. The result now follows from \ref{thm: all weights of simple modules polynomial iff weight belongs to Xr(D)+p^r P(D)}.
\\
(d) The $M_i$ form a partition of $\{1, \ldots, n\}$. By definition of $G$, we have $W = S_{M_1} \times \ldots \times S_{M_l} \subseteq W' = S_n$ and the $d_i$ form a basis of $X_0(T)$. Thus, \ref{thm: construction of phi} shows that $G$ has the properties of \ref{assumption: existence of function}, so that the result follows from \ref{cor: criterion for simple modules with Phi and d}.  

\end{proof}

We give an example showing that our approach does not work for the connected component $G$ of the even orthogonal similitude group $GO_{2l}$. Let $l = 4$ and choose $p, r$ such that $4 \vert p^r - 1$. Let $T, T', d$ be as in \ref{cor: simple polynomial GrT-modules for GL_n, GSp_2l, GO2_l+1}.(b). Let  $\varphi: X(T') \rightarrow \Z$, $t \mapsto \sum_{i=1}^l \min \{t_i, t_{i'} \}$. The same proof shows that with the exception of $(iv)$, all statements in the proof of \ref{thm: construction of phi} hold for $\varphi$ and $d$.\\
\indent Let $\lambda_0, \tilde{\lambda} \in X(T')$ be the characters given by 
\begin{align*}
\lambda_0 &=\left(\frac{p^r - 1}{2}, \frac{p^r - 1}{2}, \frac{p^r - 1}{2}, \frac{p^r - 1}{2}, \frac{p^r - 1}{4}, \frac{p^r - 1}{4}, \frac{p^r - 1}{4}, \frac{p^r - 1}{4}\right),\\
\tilde{\lambda} &= (0,0,0,0, 1, 1, -1, 1).
\end{align*} 
Then 
\begin{align*}
\varphi(\lambda_0) &= p^r - 1 > 0, \\
\varphi(\lambda_0 - p^r d) &= -1 < 0, \\
\varphi(\tilde{\lambda}) &= -1 < 0
\end{align*}
and $\lambda_0 \vert_{T} \in X_r(T)$, so that $\lambda_{0} \vert_{T} \in P_r(D), \tilde{\lambda}\vert_{T} \notin P(D)$. \\
\indent In $W \cong S_l \ltimes (\Z/(2))^{l-1}$, $S_l$ acts on $X(T')$ as in the proof of \ref{cor: simple polynomial GrT-modules for GL_n, GSp_2l, GO2_l+1}.(b) while $(\Z/(2))^{l-1}$ acts as the subgroup $\lbrace x \in (\Z/(2))^l \vert x_i \neq 0$ for an even number of $i \rbrace$ of $(\Z/(2))^{l}$. Thus, $\lambda_0$ is invariant under $S_l$ and we have $\varphi(w. \lambda_0 \vert_{T} + p^r \tilde{\lambda} \vert_{T}) > 0$ and hence $w. \lambda_0 \vert_{T} + p^r \tilde{\lambda}\vert_{T} \in P(D)$ for all $w \in (\Z/(2))^{l-1}$. Hence the arguments in the proof of \ref{thm: all weights of simple modules polynomial iff weight belongs to Xr(D)+p^r P(D)} are not applicable in this case.\\
\indent As the statement $w. \lambda_0 \vert_T + p^r \tilde{\lambda} \vert_T \in P(D)$ does not depend on $\varphi$ and there is no other choice for $d$ in this case, we cannot find other $\varphi$ or $d$ in this example to salvage the proof of \ref{thm: all weights of simple modules polynomial iff weight belongs to Xr(D)+p^r P(D)}. However, since we have not computed the weights of $\widehat{L_r}(\lambda_0 \vert_T + p^r \tilde{\lambda}\vert_T)$, we do not know whether all weights of this module are in fact polynomial.

\section{Equivalences between blocks of polynomial representations}

We adopt the notation of the previous section. Suppose that $G$ has simply connected derived subgroup and that there is $0 \neq b \in X_0(T) \cap P(D)$ such that $\langle b \rangle = X_0(T)$. Choose representatives $(\omega_\alpha)_{\alpha \in \Delta}$ of the fundamental dominant weights of the derived subgroup such that for all $\alpha \in \Delta$, $\omega_\alpha \in P(D)$ and $\omega_\alpha - b \notin P(D)$. Then the $\omega_\alpha$ together with $b$ form a basis of $X(T)$. Let $\varphi: X(T) \rightarrow \Z$ be the projection on the coordinate corresponding to $b$ and define $P_r(D) = \lbrace \lambda \in X_r(T) \cap P(D) \vert 0 \leq \varphi(\lambda) \leq p^r -1\rbrace$. Then each $\lambda \in X(T)$ can be written uniquely as $\lambda = \lambda_0 + p^r \tilde{\lambda}$ with $\lambda_0 \in P_r(D), \tilde{\lambda} \in X(T)$. Note that this redefinition of $P_r (D)$ is compatible with the old definition for $G = \GL_n, GSP_{2l}, GO_{2l+1}$.\\
\indent By \cite[II.1.18]{Jantz1}, we have $X(G) \cong X_0(T)$ via restriction, so that we can regard $b$ as a character of $G$ and have a shift functor $[b]:\m G_r T \rightarrow \m G_r T, V \mapsto V[b] = V \otimes_k k_b$, where $k_b$ is the one-dimensional $G$-module induced by $b$. This shift functor is an autoequivalence of $\m G_r T$ with inverse $[-b]$, hence it induces equivalences between blocks of $\m G_r T$.\\
\indent Since $-b$ is not a polynomial weight, $[-b]$ does not restrict to an autoequivalence of the category of polynomial $G_r T$-modules. However, we will see in this section that the functor $[b]$ sometimes still induces an equivalence between blocks of this category.\\
\indent Let $W_p \cong W \ltimes p\Z R$ be the affine Weyl group of $G$, let $R^+$ be a set of positive roots and $\rho = \sum_{\alpha \in R^+} \frac{1}{2}\alpha$ be the half-sum of positive roots. We denote by $w \bullet \lambda = w(\lambda + \rho) -\rho$ the dot-action of $w \in W_p$ on $\lambda \in X(T)$. For $a \in \Z$, if $a = qp + r$ for $q,r \in \Z$ such that $0 \leq r \leq p - 1 $, we write $ r = a \mod p$.      

\begin{lemma}
\label{lem: shift defines bijection between orbits of the affine weyl group}
Let the derived subgroup of $G$ be simply connected. Let $\lambda \in P_r(D) + p^r P(D)$ and $a = \max_{w \in W}\{ \varphi(w.\lambda+ w.\rho - \rho) \mod p\}$. Then for $1 \leq i \leq p-a-1$, the shift $[ib]$ defines a bijection
\begin{equation*}
W_p \bullet \lambda \cap (P_r(D)+p^r P(D)) \rightarrow W_p \bullet(\lambda + i b) \cap (P_r(D)+p^r P(D)).
\end{equation*}
\end{lemma}

\begin{proof}
Let $(w, pl \mu) \in W_p$. We show that $(w, pl \mu) \bullet \lambda \in P_r(D) + p^r P(D)$ iff $(w, pl \mu)\bullet(\lambda + i b) \in P_r(D) + p^r P(D)$. We have $(w, pl \mu)\bullet \lambda = w.\lambda + w.\rho - \rho + pl\mu$ and $(w, pl \mu)\bullet(\lambda+ ib) = (w, pl \mu)\bullet \lambda  + i b$ by definition of the dot-action and since $w.b = b$. Writing $ w.\lambda + w.\rho - \rho + pl\mu = \nu_0 + p^r \tilde{\nu}$ with uniquely determined $\nu_0 \in P_r(D)$ and $\tilde{\nu} \in X(T)$, we have $\varphi(\nu_0) \mod p = \varphi(w.\lambda+ w.\rho - \rho) \mod p \leq a$. It follows that $(\varphi(\nu_0) \mod p) + i \leq a + p - a - 1 = p - 1$. Thus, $\varphi(\nu_0 + ib) \leq p^r - 1$, so that $\nu_0 + i b \in P_r(D)$. It follows that $(w, pl \mu)\bullet(\lambda+ ib) = (\nu_0 + ib) + p^r \tilde{\nu}$ is the unique decomposition of $(w, pl \mu)\bullet(\lambda+ ib)$ as a sum of elements of $P_r (D)$ and $p^r X(T)$. We get $(w, pl \mu)\bullet \lambda \in P_r(D) + p^r P(D) \Longleftrightarrow \tilde{\nu} \in P(D) \Longleftrightarrow (w, pl \mu)\bullet (\lambda + ib) \in P_r(D) + p^r P(D)$.
\end{proof}

For $G = \GL_n$, we can choose $b = \det$ and have $a = \max_{w \in W}\{ (w.\lambda)_n + (w.\rho)_n - \rho_n \mod p\}$. We also have that every $G_r T$-module $V$ such that all weights of $V$ are polynomial lifts to $\overline{G_r T} = M_r D$ by Jantzen's result in \cite[5.4]{Nak1}, implying $\Ext_{G_r T}^1(V, W) \cong \Ext_{M_r D}^1(V, W)$ for all $V, W \in \m M_r D$. For large $r$, the infinitesimal Schur algebra $S_d(G_r T)$ is isomorphic to $S_d(G)$, so that the following proposition can also be applied to blocks of the ordinary Schur algebra of $\GL_n$. 
 
\begin{proposition}
 Let $G = \GL_n, T$ the torus of diagonal matrices. Let $\lambda \in P_r(D) + p^r P(D)$ and $a = \max_{w \in W}\{ (w.\lambda)_n + (w.\rho)_n - \rho_n \mod p\}$. Then for $1 \leq i \leq p-a-1$, the functor $[i \det]$ defines an equivalence between the blocks of $\m M_r D$ containing $\widehat{L}_r(\lambda)$ and $\widehat{L}_r(\lambda + i \det)$. 
\end{proposition}

\begin{proof}
Let $B$ be the block containing $\lambda$ and $B'$ be the block containing $\lambda + i \det$. As $[i \det]$ maps nonsplit extensions in $\m M_r D$ to nonsplit extensions in $\m M_r D$, we have $B[i \det]\subseteq B'$.\\
\indent Let $\mu \in B'$. Suppose there is $\nu \in B[i \det]$ such that $\Ext^1_{M_r D}(\widehat{L_r}(\mu), \widehat{L_r}(\nu)) \neq 0$. Then $\mu \in W_p\bullet(\lambda + i \det)$ since $B' \subseteq W_p\bullet(\lambda + i \det)$ by \cite[II.9.19]{Jantz1}. Now \ref{lem: shift defines bijection between orbits of the affine weyl group} yields $\mu - i\det \in P_r(D)+p^r P(D)$. Applying $[-i\det]$ to a nonzero element of $\Ext^1_{M_r D}(\widehat{L_r}(\mu), \widehat{L_r}(\nu))$, we get $\Ext^1_{G_r T}(\widehat{L_r}(\mu - i\det), \widehat{L_r}(\nu - i\det)) \neq 0$. Since both simple modules lift to $M_r D$, Jantzen's result yields $\Ext^1_{G_r T}(\widehat{L_r}(\mu-i\det), \widehat{L_r}(\nu-i\det)) \cong \Ext^1_{M_r D}(\widehat{L_r}(\mu-i\det), \widehat{L_r}(\nu-i\det)) \neq 0$, so that $\mu - i\det \in B$ and $\mu \in B[i\det]$.
\end{proof}
We do not know whether an analogue of \cite[5.4]{Nak1} and hence a generalization of the foregoing result holds for other reductive group schemes. 

\section*{Acknowledgment}
The results of this article are part of my PhD thesis which I am currently writing at the University of Kiel. I would like to thank my advisor Rolf Farnsteiner for his support and helpful discussions as well as the members of my research team for proofreading.
\bibliographystyle{abbrv}
\bibliography{schur2}
%\printbibliography

\end{document}